\documentclass{amsart}

\usepackage{amsmath}
\usepackage{amssymb}
\usepackage{datetime}
\usepackage{enumitem}

\allowdisplaybreaks

\newtheorem{theorem}{Theorem}[section]
\newtheorem{lemma}[theorem]{Lemma}

\theoremstyle{definition}
\newtheorem{definition}[theorem]{Definition}

\theoremstyle{remark}
\newtheorem{remark}[theorem]{Remark}

\numberwithin{equation}{section}

\newcommand{\Z}{\mathbb{Z}}
\newcommand{\N}{\mathbb{N}}
\newcommand{\OO}{\mathcal{O}}
\newcommand{\iO}{\mathcal{JO}}
\newcommand{\F}{\mathfrak{F}}

\newcommand{\C}{\mathbb{C}}
\newcommand{\B}{\mathbb{B}}
\newcommand{\PC}{\mathbb{P}^n(\mathbb{C})}

\newcommand{\R}{\mathbb{R}}
\newcommand{\T}{\mathbb{T}}
\newcommand{\bS}{\mathbb{S}}
\newcommand{\lcm}{\mathrm{lcm}}

\newcommand{\mA}{\mathcal{A}}
\newcommand{\mR}{\mathcal{R}}
\newcommand{\mH}{\mathcal{H}}
\newcommand{\mHR}{\mathcal{HR}}
\newcommand{\mT}{\mathcal{T}}

\newcommand{\Lie}{\mathrm{Lie}}

\newcommand{\sty}{\displaystyle}

\begin{document}

\title[Toeplitz operators on some Weakly Pseudoconvex Domains]{Toeplitz Operators
with Quasi-Homogeneuos Quasi-Radial Symbols on some Weakly Pseudoconvex Domains}

\author{Raul Quiroga-Barranco}
\address{Centro de Investigaci\'on en Matem\'aticas \\
Guanajuato \\
Mexico}
\email{quiroga@cimat.mx}

\author{Armando Sanchez-Nungaray}
\address{Facultad de Matem\'aticas \\
Universidad Veracruzana \\
Veracruz \\
Mexico}
\email{armsanchez@uv.mx}

\begin{abstract}
	On the weakly pseudo-convex domains $\Omega_p^n$ we introduce
	quasi-homogeneous quasi-radial symbols. These are used to prove
	the existence of a commutative Banach algebra of Toeplitz operators
	on Bergman space of $\Omega_p^n$. We also show that group theoretic
	and geometric properties for our symbols are satisfied. The results
	presented here contain the geometric description of the symbols
	introduced by N.~Vasilevski in \cite{NikolaiQuasi} for the
	unit ball $\B^n$.
\end{abstract}

\thanks{The authors were supported by SNI and a Conacyt grant.}

\maketitle

\section{Introduction}
In the last years, there has been an interest in the study and classification of
commutative algebras of Toeplitz operators. This is accompanied by the study of
the underlying geometric structures associated to the corresponding
symbols. Some previous works along these lines can be found in the references.
In this work, we continue the investigation of such topic by introducing
quasi-homogeneous quasi-radial symbols on a family of weakly pseudo-convex
domains $\Omega_p^n$ (see the definitions below). Such family of domains
contains the unit ball $\B^n$ as a particular case. Furthermore, our symbols
generalize those considered in \cite{NikolaiQuasi} for the unit ball
$\B^n$ as well as those considered in \cite{QS-Quasi-PC} for the complex
projective space $\PC$. The sets of symbols introduced in this work
are shown to define commutative Banach algebras of Toeplitz operators, which
are not $C^*$, in the Bergman space of $\Omega_p^n$.
In \cite{QV-Reinhardt} it was shown the existence of ``large''
commutative $C^*$-algebras of Toeplitz on Reinhardt domains. With this respect,
the commutative algebras of Toeplitz operators constructed in this
work are the first Banach not $C^*$ examples defined on Reinhardt domains that
are not bounded symmetric.

As in most of the works that precede this one, it turns out that the symbols
that define our commuting Toeplitz operators have a strong geometric
background. In fact, we prove that our quasi-homogeneous quasi-radial symbols
have an associated group of isometries characterized as the largest connected
subgroup leaving invariant the symbols; this is among the toral isometries
of $\Omega_p^n$ as Reinhardt domain (see Theorem~\ref{thm:assoc-abelian-group}).
We also obtain Lagrangian frames for the set of symbols which are defined
on the leaves of a foliation of $\Omega_p^n$ (see Theorem~\ref{thm:assoc-Lag-frame}).

It is worthwhile to mention that this work provides new results for
the case of the unit ball $\B^n$. More precisely, since our quasi-homogeneous
quasi-radial symbols for $\Omega_p^n$ generalize those found in \cite{NikolaiQuasi},
our Theorems~\ref{thm:assoc-abelian-group} and \ref{thm:assoc-Lag-frame}
hold for the symbols considered in \cite{NikolaiQuasi} for $\B^n$
(see Remark~\ref{rmk:unit-ball}). Hence, this work establishes the geometric
nature and properties for the quasi-homogeneous quasi-radial symbols introduced
by N.~Vasilevski in \cite{NikolaiQuasi}.

\section{Preliminaries}
Given a multi-index $\alpha=(\alpha_1,\alpha_2,\ldots,\alpha_n) \in \N^n$
we will use the standard notation
\begin{align*}
	|\alpha| &= \alpha_1 + \alpha_2 + \cdots+ \alpha_n,\\
	\alpha! &= \alpha_1 !  \alpha_2 ! \cdots \alpha_n !,\\
	z^\alpha &= z_1^{\alpha_1}z_2^{\alpha_2}\cdots z_n^{\alpha_n}.
\end{align*}

For $p \in \Z_+^n$ a fixed multi-index, we define the following sets
\begin{align}
	\Omega_{p}^{n}(r)&=
		\Big\{(z_1,\ldots,z_n)\in\C^{n}\,|\,\sum_{j=1}^{n}|z_j|^{2p_j}<r^{2}\Big\}\\
	\bS_{p}^{n}(r)&=
		\Big\{(z_1,\ldots,z_n)\in\C^{n}\,|\,\sum_{j=1}^{n}|z_j|^{2p_j}=r^{2}\Big\}
\end{align}
For the case $r=1$, we simply write $\Omega_{p}^{n}$ and $\bS_{p}^{n}$,
respectively.
Note that for $p=(1,\dots,1)$, we have that $\Omega_{p}^{n}=\B^n$ is
the unit ball and $\bS_p^n=\bS^n$ is the unit sphere in $\C^n$,
both centered at the origin.

For every $z \in \C^n$, we also denote
\begin{align*}
	r &= \|z\|_p = \sqrt{|z_1|^{2p_1}+\cdots+|z_n|^{2p_n}} \\
	\xi_j &= \frac{z_j}{\|z\|_p^\frac{1}{p_j}}
		= \frac{z_j}{r^\frac{1}{p_j}}
\end{align*}
for all $j=1,\dots,n$. In particular we have
\[
	\sum_{j=1}^{n}|\xi_j|^{2p_j} = 1,
\]
which implies that $\xi=(\xi_1,\dots,\xi_n)\in \bS_{p}^{n}$. Note that
these expressions define a set of coordinates $(r,\xi)$ for every
$z \in \C^n$, these coordinates are called $p$-polar coordinates.

We let $dv$ denote the Lebesgue measure on $\Omega_{p}^{n}$
normalized so that $v(\Omega_{p}^{n})=1$.
Also we let $\sigma$ denote the hypersurface measure on $\bS_{p}^{n}$,
also normalized so that $\sigma(\bS_p^n)=1$.

It is well known that the measures on these sets are related by
the following identity, that holds in the case of $\B^n$ and $\bS^n$,
\[
	\int_{\B^{n}} f(z)dv(z)=2 n \int_0^1 r^{2n-1}
	\int_{\bS^n}f(r,\xi)d\sigma(\xi)dr
\]
for every non-negative measurable function $f$ on $\B^n$.
For more details we refer the reader to \cite{Zhu}.
The next result is an
extension of this equation for arbitrary $p$ as above.

\begin{lemma}
	The measures on $v$ and $\sigma$ satisfy
	\[
		\int_{\Omega_{p}^{n}} f(z)dv(z) =
			\left(2\sum_{j=1}^{n}\frac{1}{p_j}\right)
			\int_0^1 r^{2(\sum_{j=1}^{n}\frac{1}{p_j})-1}
			\int_{\bS_p^n}f(r,\xi)d\sigma(\xi)dr.
	\]
	for every non-negative measurable function $f$ on $\Omega_p^n$.
\end{lemma}
\begin{proof}
	The proof follows the same arguments as those found in \cite{Zhu}
	for $\B^n$ and $\bS^n$, and we present them for the sake of
	completeness.
	
	Denote by $dV = dx_1dy_1\cdots dx_ndy_n$ the $n$-dimensional
	Lebesgue measure without normalization.
	Similarly, let $dS$ denote the hypersurface measure of
	$\bS_p^n$ before normalization. As usual, these coordinates are
	obtained by the identification $z_j = x_j+iy_j$.
	Then, the Lebesgue measure of the solid given by
	by $dS$ in $\bS_p^n$ together with $r>0$ and $r+dr$, is given by
	\[
		dV=\frac{dS}{S(1)}\big(V(r+dr)-V(r)\big)
	\]
	where $V(r)$ in the Lebesgue measure of $\Omega_p^n(r)$
	and $S(r)$ is the is the measure of $\bS_p^n(r)$.

	The change of variables $z_j=r^\frac{1}{p_j}\xi_j$ yields
	\begin{equation}\label{Volumen-r}
		V(r)=\int_{\Omega_{p}^{n}(r)}dV(z)=
		r^{2\sum_{j=1}^{n}\frac{1}{p_j}}V(1),
	\end{equation}
	and so it follows that
	\[
		dV = \frac{V(1)dS}{S(1)}\left((r+dr)^{2\sum_{j=1}^{n}
			\frac{1}{p_j}}-r^{2\sum_{j=1}^{n}\frac{1}{p_j}}\right).
	\]	
	This proves that
	\[
		dV=\frac{V(1)}{S(1)}\left(2\sum_{j=1}^{n}
			\frac{1}{p_j}\right)r^{2(\sum_{j=1}^{n}\frac{1}{p_j})-1}drdS
	\]
	which clearly implies
	\[
		dv=\left(2\sum_{j=1}^{n}
			\frac{1}{p_j}\right)r^{2(\sum_{j=1}^{n}\frac{1}{p_j})-1}drd\sigma
	\]
	where $dv$ and $d\sigma$ are the normalized measures of $\Omega_{p}^{n}$
	and $\bS_{p}^{n}$, respectively.
\end{proof}

The Hilbert spaces $L^2(\Omega_{p}^{n})$ and $L^2(S_p^n)$
are those associated to the usual Lebesgue measure $dV$ on
$\Omega_{p}^{n}$ and the hypersurface
measure $dS$ on $S_p^n$.
We denote by $\mA^2(\Omega_{p}^{n})$ the closed
subspace of $L^2(\Omega_{p}^{n})$ consisting of those
functions which are holomorphic in $\Omega_{p}^{n}$,
and we let $P_{p}:L^2(\Omega_{p}^{n}) \rightarrow
\mA^2(\Omega_{p}^{n})$ be the
orthogonal projection. If $a\in L_{\infty}(\Omega_{p}^{n})$
then the Toeplitz operator $T_a$ with symbol $a$ is the bounded
operator on $\mA^2(\Omega_{p}^{n})$ defined by $T_a(f) = P_p(af)$.

We have following identity, and for its proof
we refer to \cite{C-R}.
\begin{equation}\label{inner-product-monomials}
	\langle z^{\alpha},z^{\beta} \rangle =
	\delta_{\alpha,\beta}  \frac{\pi^n\sty\prod_{j=1}^{n}
	\Gamma\left(\frac{\alpha_j+1}{p_{j}}\right)}{\sty\left[\prod_{j=1}^{n}p_{j}\right]
	\Gamma\left(\sum_{j=1}^{n}\frac{\alpha_j+1}{p_{j}}+1\right)}.
\end{equation}
In particular, we also have
\begin{equation}\label{Volumen-Op}
	C_{p}^{n} := V(1) = \frac{\pi^n\sty\prod_{j=1}^{n}
	\Gamma\left(\frac{1}{p_{j}}\right)}{\sty\left[\prod_{j=1}^{n}p_{j}\right]
	\Gamma\left(\sum_{j=1}^{n}\frac{1}{p_{j}}+1\right)}
\end{equation}
As a consequence of \eqref{inner-product-monomials} we obtain an orthonormal
basis for $\mA^2(\Omega_{p}^{n})$ given by
\begin{equation}\label{eq:orth-basis}
  \left( \frac{\sty\left[\prod_{j=1}^{n}p_{j}\right]
  \Gamma\left(\sum_{j=1}^{n}\frac{\alpha_j+1}{p_{j}}+1\right)}{\pi^n\sty\prod_{j=1}^{n}
  \Gamma\left(\frac{\alpha_j+1}{p_{j}}\right)} \right)^\frac{1}{2} z^{\alpha}
\end{equation}

When $p=(1,\ldots,1)$ and for every $\alpha,\beta \in \N^n$
we have
\[
	\int_{\bS^n} \xi^{\alpha} \overline{\xi^{\beta}}dS(\xi) =
	\delta_{\alpha, \beta}\frac{2 \pi^n \alpha!}{(n-1+|\alpha|)!}.
\]
For a proof of this identity we refer to \cite{Zhu}. Such identity has the
the following extension for arbitrary $p\in \N^n$.

\begin{lemma}
	Let $p \in \Z_+^n$ and $\alpha, \beta \in \N^n$. Then
	\begin{equation}\label{integralsurface1}
		\int_{\bS_{p}^{n}} \xi^{\alpha} \overline{\xi^{\beta}}
			d\sigma(\xi) = \delta_{\alpha,\beta}
		\frac{\sty\Gamma\left(\sum_{j=1}^{n}\frac{1}{p_j}\right)
		\prod_{j=1}^{n}\Gamma\left(\frac{\alpha_j+1}{p_j}\right)}{\sty \left[\prod_{j=1}^{n}\Gamma\left(\frac{1}{p_j}\right)\right]
		\Gamma\left(\sum_{j=1}^{n}\frac{\alpha_j+1}{p_j}\right)},
	\end{equation}
	or equivalently
	\begin{equation} \label{integralsurface2}
		\int_{\bS_{p}^{n}}\xi^{\alpha} \overline{\xi^{\beta}} dS(\xi) = \delta_{\alpha,\beta}
		\frac{\sty 2 \pi^n \prod_{j=1}^{n}\Gamma\left(\frac{\alpha_j+1}{p_j}\right)}
		{\sty \left[\prod_{j=1}^{n}p_j \right]\Gamma\left(\sum_{j=1}^{n}
		\frac{\alpha_j+1}{p_j}\right)}
	\end{equation}
\end{lemma}
\begin{proof}
	As a consequence of the rotation invariance of the measure on $\bS^n_p$
	it follows that
	\[
		\int_{\bS_{p}^{n}}\xi^{\alpha} \overline{\xi^{\beta}}d\sigma(\xi)=0.
	\]
	for every $\alpha \not= \beta$. Following \cite{Zhu}, we evaluate
	the integral
	\begin{equation*}
		I = \int_{\C^n}|z^{\alpha}|^{2}e^{-(|z_{1}|^{2p_{1}}+\cdots+|z_{n}|^{2p_{n}})}dV
	\end{equation*}
	by two diferents methods. First, Fubini's theorem gives
	\begin{align*}
	I &=\int_{\C^{n}}|z^{\alpha}|^{2}
		e^{-\left(|z_{1}|^{2p_{1}}+\cdots+|z_{n}|^{2p_{n}}\right)}dV  \\
	 &=\prod_{j=1}^{n}\int_{\R^{n}}(x^{2}+y^{2})^{\alpha_j}e^{-(x^2+y^2)^{p_j}}dxdy  \\
	 &=\pi^{n}\prod_{j=1}^{n}\int_{0}^{\infty}r^{\alpha_j}e^{-r^{p_j}}dr.
	\end{align*}

	Using the change of variables $r=\sty s^{\frac{1}{p_j}}$, for which
	$dr=\sty\frac{1}{p_j}s^{\frac{1}{p_j}-1}ds$, we obtain
	\begin{align} \label{aux1}
		I&=\pi^{n}\prod_{j=1}^{n}\frac{1}{p_j}\int_{0}^{\infty}
			s^{  \frac{\alpha_j+1}{p_j}   -1 }e^{-s}ds \notag \\
		&= \pi^{n}\prod_{j=1}^{n}\frac{1}{p_j}\Gamma\left(\frac{\alpha_j+1}{p_j}\right)
	\end{align}

	On the other hand, we integrate in $p$-polar cordinates
	\begin{equation*}
		I = 2\sum_{j=1}^{n}\frac{1}{p_j}C_{p}^{n}\int_{0}^{\infty}
		r^{2\left(\sum_{j=1}^{n}\frac{\alpha_j+1}{p_j}\right)-1}
		e^{-r^{2}}dr\int_{\bS_{p}^{n}}|\xi^{\alpha}|^{2}d\sigma(\xi)
	\end{equation*}
	where $\xi_j=\sty\frac{z_{j}}{r^{\frac{1}{p_{j}}}}$, for all $j=1,\dots,n$.
	If we take $r=\sqrt{s}$, then $dr=\frac{1}{2}s^{-\frac{1}{2}}ds$, and so
	\[
		I = 2\left(\sum_{j=1}^{n}\frac{1}{p_j}\right) C_{p}^{n}\int_{0}^{\infty}s^{\left(\sum_{j=1}^{n}\frac{\alpha_j+1}{p_j}\right)-1} \frac{ds}{2}\int_{\bS_{p}^{n}}|\xi^{\alpha}|^{2}d\sigma(\xi)
	\]
	which is equivalent to
	\begin{equation} \label{aux2}
		I = \sum_{j=1}^{n}\frac{1}{p_j}C_{p}^{n}
			\Gamma\left(\sum_{j=1}^{n}\frac{\alpha_j+1}{p_j}\right)
			\int_{\bS_{p}^{n}}|\xi^{\alpha}|^{2}d\sigma(\xi).
	\end{equation}
	From \eqref{aux1} and \eqref{aux2} we obtain the identity
	\begin{equation*}
		\pi^{n}\prod_{j=1}^{n}\frac{1}{p_j}
			\Gamma\left(\frac{\alpha_{k}+1}{p_j}\right) =
			\sum_{j=1}^{n}\frac{1}{p_j}C_{p}^{n}
			\Gamma\left(\sum_{j=1}^{n}\frac{\alpha_{k}+1}{p_j}\right)
			\int_{\bS_{p}^{n}}|\xi^{\alpha}|^{2}d\sigma(\xi).
	\end{equation*}
	Hence, we conclude that
	\begin{align*}
		\int_{\bS_{p}^{n}}|\xi^{\alpha}|^{2} d\sigma(\xi)
		&=\frac{\sty\pi^{n}\prod_{j=1}^{n} \left[\frac{1}{p_j}
		\Gamma\left(\frac{\alpha_j+1}{p_j}\right)\right]}
		{\sty C_{p}^{n}\left(\sum_{j=1}^{n}\frac{1}{p_j}\right)
		\Gamma\left(\sum_{j=1}^{n}\frac{\alpha_j+1}{p_j}\right)} \\
		&= \frac{\sty\left(\prod_{j=1}^{n}p_j\right)
		\Gamma\left(\sum_{j=1}^{n}\frac{1}{p_j}+1\right)
		\prod_{j=1}^{n}\left[\frac{1}{p_j}
		\Gamma\left(\frac{\alpha_j+1}{p_j}\right)\right]}
		{\sty\left[\prod_{j=1}^{n}\Gamma\left(\frac{1}{p_j}\right)\right]
		\left(\sum_{j=1}^{n}\frac{1}{p_j}\right)
		\Gamma\left(\sum_{j=1}^{n}\frac{\alpha_j+1}{p_j}\right)} \\
		&=\frac{\sty\Gamma\left(\sum_{j=1}^{n}\frac{1}{p_j}\right)
		\prod_{j=1}^{n}\Gamma\left(\frac{\alpha_j+1}{p_j}\right)}{\sty \left[\prod_{j=1}^{n}
		\Gamma\left(\frac{1}{p_j}\right)\right]
		\Gamma\left(\sum_{j=1}^{n}\frac{\alpha_j+1}{p_j}\right)},
	\end{align*}
	which proves \eqref{integralsurface1}.
	
	We now recall that the volume of $\Omega_p^n(r)$ is
	given by \eqref{Volumen-r} and that
	\[
		V(r)=\int_0^r S(x)dx
	\]
	where $S(x)$ is the volume of $\bS_{p}^{n}(x)$. It follows that
	\[
		S(r) = V'(r) = \left(2\sum_{j=1}^{n}\frac{1}{p_j}\right)
		r^{2\sum_{j=1}^{n}\frac{1}{p_j}-1}V(1)
	\]
	In particular the volume of $\bS_{p}^{n}$ is given by
	\[
		S(1) = \left(2\sum_{j=1}^{n}\frac{1}{p_j}\right)V(1) =
		\frac{2\pi^n\sty\prod_{j=1}^{n}\Gamma\left(\frac{1}{p_{j}}\right)}
		{\sty\left[\prod_{j=1}^{n}p_{j}\right]
		\Gamma\left(\sum_{j=1}^{n}\frac{1}{p_{j}}\right)}
	\]
	Using this identity and \eqref{integralsurface1}
	we obtain \eqref{integralsurface2}.
\end{proof}

\section{Toeplitz operators with quasi-homogeneuos symbols}
Let $k=(k_1,\ldots, k_s) \in \Z_+^s$ be a partition of $n$; in other words we have
$|k| = k_1 + \ldots + k_s = n$. Given such $k$ we define
\[
	\widehat{k}_{j}=
	\begin{cases}
		0 & \text{ if } j = 0, \\
		\widehat{k}_{j-1} + k_j & \text{ if } j=1, \dots ,s
	\end{cases}
\]

Given $z \in \C^n$ and $k \in \Z_+^s$ a partition of $n$ we decompose $z$ into
$s$ pieces, each one of which has $k_j$ components for $j=1,\ldots, s$. This is
achieved by denoting
\[
	z_{(j)}=(z_{\widehat{k}_{j-1}},\dots, z_{\widehat{k}_j}),
\]
for $j = 1, \dots, s$. In particular, $z_{(j)} \in \C^{k_j}$ for every
$j = 1, \dots s$ and $z = (z_{(1)}, \dots z_{(s)})$ for
every $z \in \C^n$. We also note that the components of $z_{(j)} \in \C^{k_j}$ are
given by $(z_{(j)})_1, \dots, (z_{(j)})_{k_j}$.

Now let us choose $p \in \Z_+^n$ as before. Then, we denote
\begin{equation} \label{radial-components}
	r_{j}=\|z_{(j)}\|_p=
		\sqrt{\sum_{t=\widehat{k}_{j-1}+1}^{\widehat{k}_{j}}|z_{t}|^{2p_{t}}}
\end{equation}
for every $j = 1, \dots, s$. Furthermore, for every $z \in \C^n$ we denote
\[
	\xi_{(j)} =
	\left(\frac{(z_{(j)})_1}{r_j^\frac{1}{(p_{(j)})_1}}, \dots,
	\frac{(z_{(j)})_{k_j}}{r_j^\frac{1}{(p_{(j)})_{k_j}}}\right),
\]
for every $j = 1, \dots, s$. In words, the element $\xi_{(j)}$ is obtained
from $z_{(j)}$ by dividing each component of the latter by
$r_j = \|z_{(j)}\|_p$ powered to the exponents given by the reciprocals
of the corresponding components of $p_{(j)}$. Note that
$\xi_{(j)} \in \bS^{k_j}_{p_{(j)}}$ for every $j$.

\begin{definition}
  \label{def:quasi-homogeneous}
  Let $k = (k_1, \dots, k_s) \in \Z_+^{s}$ be a partition of $n$
  and let $\nu,\mu \in \N^{n}$ be such that
  \[
  \nu \cdot \mu = \nu_1\mu_0 + \dots + \nu_n\mu_n = 0.
  \]
  With the above notation, the $k$-quasi-homogeneous symbol associated
  to $\nu,\mu$ is the function $\varphi : \Omega_p^n \rightarrow \C$ given by
  \[
  \varphi(z) = \xi^{\nu}\overline{\xi}^{\mu} = \prod_{j=0}^s
  \left(\xi_{(j)}\right)^{\nu_{(j)}}
  \left(\xi_{(j)}\right)^{\mu_{(j)}}.
  \]
  We will denote by $\mH_k(\Omega_p^n)$ the set of $k$-quasi-homogeneous symbols
  on $\Omega_p^n$.
\end{definition}

It is a simple matter to see that for $p=(1,\dots,1)$ a $k$-quasi-homogeneous
symbol on $\Omega_p^n = \B^n$ is $k$-quasi-homogeneous in the sense of
\cite{NikolaiQuasi}.

\begin{definition}
  \label{def:quasi-radial}
  Let $k = (k_1, \dots, k_s) \in \Z_+^{s}$ be a partition of
  $n$. With the above notation, a $k$-quasi-radial symbol is a
  function function $a : \Omega_p^n \rightarrow \C$ that can be written as
  $a(z) = \widetilde{a}(r_1, \dots, r_s)$
  where $r_j$ is given by \ref{radial-components}.
  We will denote by $\mR_k(\Omega_p^n)$ the set of
  $k$-quasi-radial symbols on $\Omega_p^n$.
\end{definition}

As before, it is easy to see that for $p=(1,\dots,1)$ a $k$-quasi-radial
symbol on $\Omega_p^n = \B^n$ is $k$-quasi-radial in the sense of
\cite{NikolaiQuasi}

The two definitions above yield together the notion of
quasi-homogeneous quasi-radial symbol.

\begin{definition}
  \label{def:quasi-homog-radial}
  Let $k = (k_1, \dots, k_s) \in \Z_+^{s}$ be a partition of
  $n$. Then, a $k$-quasi-homogeneous quasi-radial symbol is a
  function $\Omega_p^n \rightarrow \C$ of the form $a\varphi$, where $a \in
  \mR_k(\Omega_p^n)$ and $\varphi \in \mH_k(\Omega_p^n)$; in this case, we will refer to $a$
  and $\varphi$ as the quasi-radial and quasi-homogeneous parts of the
  symbol, respectively. We will denote by $\mHR_k(\Omega_p^n)$ the set of
  $k$-quasi-homogeneous quasi-radial symbols.
\end{definition}

We observe that the domain $\Omega_{p}^{n}$ is a Reinhardt domain and so
the results from \cite{QV-Reinhardt} can be applied in this case.
In particular, \cite{QV-Reinhardt} implies that the set symbols
$\mR_{(1,\dots,1)}$ defines a family of commutative Toeplitz
operators on $\mA^2(\Omega_{p}^{n})$. Furthermore,
every Toeplitz operator $T_a$ with symbol $a \in\mR_k(\Omega_p^n)$ is diagonal with
respect to the standard monomial basis. Since $\mR_k(\Omega_p^n)$ is a subset
of $\mR_{(1,\dots,1)}(\Omega_p^n)$, the same properties hold for $\mR_k(\Omega_p^n)$.
The following result provides an explicit presentation of
such Toeplitz operators as diagonalizable operators.

\begin{lemma} \label{Lemma-kradial}
  Let $k \in \Z_+^s$ be a partition of $n$.
  Then, for any $k$-quasi-radial bounded measurable symbol
  $a \in \mR_k(\Omega_p^n)$, we have
  \[
	  T_a z^{\alpha}=\gamma_{a,k}(\alpha)z^{\alpha},
  \]
  for every $\alpha\in \N^n$, where
  \begin{align}
   \label{spectrum-radial}
    \notag
    \gamma_{a,k}(\alpha)& =
    \frac{\sty 4^{s}\Gamma\left(\sum_{j=1}^{n}\frac{\alpha_{j}+1}{p_{j}}+1\right)
    \prod_{j=1}^{s} \left(\sum_{t=\widehat{k}_{j-1}+1}^{\widehat{k}_{j}}\frac{1}{p_{t}}\right)}
    {\sty\prod_{j=1}^{s}\Gamma\left(\sum_{t=\widehat{k}_{j-1}+1}^{\widehat{k}_{j}}   \frac{\alpha_{t}+1}{p_{t}}\right)
    }\\
    & \times \int_{\Delta_{s}^{n}} a(r_1,\ldots, r_s)\prod_{j=1}^{s}
    r_{j}^{\left(2\sum_{t=\widehat{k}_{j-1}+1}^{\widehat{k}_{j}}\left(\frac{\alpha_{t}+1}{p_{t}}\right)-1\right)}dr_{j}
  \end{align}
and $\Delta_{s}^{n}=\{ (r_1,\ldots, r_s) \in\R_+ : r_1^2+\cdots+ r_s^2<1\}$.
\end{lemma}
\begin{proof}
	Let $\alpha\in \N^n$. Then, we have
	\[
	\langle T_a z^\alpha,z^\alpha\rangle =
	\langle a(r_1,\ldots, r_s)z^\alpha,z^\alpha
	\rangle=\int_{\Omega_{p}^{n}}a(r_1,\ldots, r_s)|z^\alpha|^{2}dV
	\]

	Consider the change of variables $(z_{(j)})_t = r_j^{\frac{1}{p_{(j),t}}} (\xi_{(j)})_t$,
	for every $t = 1, \dots, k_j$ and $j = 1, \dots, s$. Then, we obtain
	\begin{align*}
		\langle a(r_1,\ldots, r_s)z^\alpha,z^\alpha\rangle &=
		\int_{\Delta_{s}^{n}} a(r_1,\ldots, r_s)\prod_{j=1}^{s}
		r_{j}^{\left[2\sum_{t=\widehat{k}_{j-1}+1}^{\widehat{k}_{j}}
		\left(\frac{\alpha_{t}+1}{p_{t}}\right)-1\right]}dr_{j}  \\
		&\times \prod_{j=1}^{s} 2\left(
		\sum_{t=\widehat{k}_{j-1}}^{\widehat{k}_{j}}\frac{1}{p_{t}}\right)
		\int_{S_{(j)}^{k_j}}|\xi_{(j)}|^{\alpha_{(j)}}dS_{j} \\
		&= \int_{\Delta_{s}^{n}} a(r_1,\ldots, r_s)\prod_{j=1}^{s}
		r_{j}^{\left[2\sum_{t=\widehat{k}_{j-1}+1}^{\widehat{k}_{j}}
		\left(\frac{\alpha_{t}+1}{p_{t}}\right)-1\right]}dr_{j}\\
		&\times \frac{\sty4^{s}\pi^{n}\prod_{j=1}^{n}
		\Gamma\left(\frac{\alpha_j+1}{p_j}\right) \prod_{j=1}^{s}
		\left(\sum_{t=\widehat{k}_{j-1}+1}^{\widehat{k}_{j}}\frac{1}{p_{t}}\right)}
		{\sty\prod_{j=1}^{s}\Gamma\left(
		\sum_{t=\widehat{k}_{j-1}+1}^{\widehat{k}_{j}}\frac{\alpha_{t}+1}{p_{t}}\right)
		\prod_{j=1}^{n}p_j},
	\end{align*}
	and the result follows from \eqref{inner-product-monomials}.
\end{proof}

We now determine the action on monomials of the Toeplitz operators
with symbols in $\mHR_k(\Omega_p^n)$.

\begin{lemma} \label{Lemma-khomogeneous}
  Let $k \in \Z_+^s$ be a partition of $n$ and $\nu, \mu \in \N^n$.
  If $a\xi^\nu\overline{\xi}^\mu = a(r_1, \dots, r_l) \xi^\nu\overline{\xi}^\mu$
  is an element of $\mHR_k(\Omega_p^n)$, then the Toeplitz operator
  $T_{a\xi^\nu\overline{\xi}^\mu}$ acts on monomials $z^{\alpha}$ with
  $\alpha\in \N^n$ as follows
  \[
  T_{a\xi^\nu\overline{\xi}^\mu} z^{\alpha}=
  \begin{cases}
    \tilde{\gamma}_{a,k,\nu,\mu}(\alpha)z^{\alpha+\nu-\mu} &
    \text{ for } \alpha + \nu - \mu \in \Z^n_+ \\
    0 & \text{ for } \alpha + \nu - \mu \not\in \Z^n_+
  \end{cases}
  \]
  where
	\begin{multline} \label{spectrum-homoge}
	 	\tilde{\gamma}_{a,k,\nu,\mu}(\alpha)
	 	= {\sty \int_{\Delta_{s}^{n}}a(r_1,\ldots,r_s)
	 	\prod_{j=1}^{s} r_{j}^{\left(\sum_{t=\widehat{k}_{j-1}+1}^{\widehat{k}_{j}}
		\frac{2\alpha_{t}+\nu_{t}-\mu_t+2}{p_{t}}-1\right)}dr_{j}} \\
		\times \frac{\sty 4^{s}\prod_{j=1}^{n}
		\Gamma\left(\frac{\alpha_{j}+\nu_{j}+1}{p_{j}}\right)
		\prod_{j=1}^{s}\left(\sum_{t=\widehat{k}_{j-1}+1}^{\widehat{k}_{j}}
		\frac{1}{p_{t}}\right)
		\Gamma\left(\sum_{j=1}^{n}\frac{\alpha_j+\nu_j-\mu_{j}+1}{p_j}+1\right)}
		{\sty \prod_{j=1}^{n} \Gamma\left(\frac{\alpha_j+\nu_j-\mu_{k}+1}{p_j}\right)
		\prod_{j=1}^{s} \Gamma\left(\sum_{t=\widehat{k}_{j-1}+1}^{\widehat{k}_{j}}
		\frac{\alpha_{t}+\nu_{t}+1}{p_{t}}\right)}
	\end{multline}
\end{lemma}
\begin{proof}
	Let $\alpha,\beta \in \N^n$. Then, we have
	\begin{align*}
		\langle T_{a\xi^{\nu}\overline{\xi}^{\mu}}z^{\alpha},z^{\beta}\rangle
		&=\langle a\xi^{\nu}\overline{\xi}^{\mu}z^{\alpha},z^{\beta}\rangle \\
		&=\int_{\Omega_{p}^{n}} a(r_1,\ldots, r_s) \xi^\nu \overline{\xi}^\mu
		z^{\alpha} \overline{z}^{\beta} dV
	\end{align*}
	
	Consider the change of variables $(z_{(j)})_t = r_j^{\frac{1}{p_{(j),t}}} (\xi_{(j)})_t$,
	for $t = 1, \dots, k_j$ and $j = 1, \dots, s$. Then, we obtain
	\begin{align*}
		\langle T_{a\xi^{\nu}\overline{\xi}^{\mu}}z^{\alpha},z^{\beta}\rangle
		&= \int_{\Delta_{s}^{n}}a(r_1,\ldots,r_s)\prod_{j=1}^{s}2
		\left(\sum_{t=\widehat{k}_{j-1}+1}^{\widehat{k}_{j}}\frac{1}{p_{t}}\right) \\
		&\quad\quad\times r_{j}^{\left(\sum_{t=\widehat{k}_{j-1}+1}^{\widehat{k}_{j}}
		\frac{\alpha_{t}+\beta_{t}+2}{p_{t}}-1\right)}dr_{j}  \\
		&\times \prod_{j=1}^{s}\int_{S_{p_{(j)}}^{k_{j}}}\xi^{\nu_{(j)}+
		\alpha_{(j)}}\overline{\xi}^{\mu_{(j)}+\beta_{(j)}}dS_j  \\
		&= \delta_{\alpha+\nu,\beta+\mu}
		\int_{\Delta_{s}^{n}}a(r_1,\ldots,r_s)\prod_{j=1}^{s}
		r_{j}^{\left[\sum_{t=\widehat{k}_{j-1}+1}^{\widehat{k}_{j}}
		\frac{\alpha_{t}+\beta_{t}+2}
		{p_{t}}-1\right]}dr_{j} \\
		&\times \frac{\sty 4^{s}\pi^{n}\prod_{t=1}^{n}
		\Gamma\left(\frac{\alpha_{t}+\nu_{t}+1}{p_{t}}\right)
		\prod_{j=1}^{s}\left(\sum_{t=\widehat{k}_{j-1}+1}^{\widehat{k}_{j}}
		\frac{1}{p_{t}}\right)}
		{\sty\prod_{j=1}^{n}p_j
		\prod_{j=1}^{s}\Gamma\left(\sum_{t=\widehat{k}_{j-1}+1}^{\widehat{k}_{j}}
		\frac{\alpha_{t}+\nu_{t}+1}{p_{t}}\right)}
	\end{align*}
	Observe that this expression is non zero if and only if $\beta =
	\alpha + \nu - \mu$, which a priori belongs to $\N^n$.
	We conclude the result from the orthonormality of the basis defined
	in \eqref{eq:orth-basis}.
\end{proof}

\section{Commutativity results for quasi-homogeneuos symbols}
\label{sec:commutativity}
We obtain commutativity results for Toeplitz operators on the domain
$\Omega_p^n$ that extend those found in \cite{NikolaiQuasi}.

\begin{theorem}
	\label{thm:commutativity-(1,k)}
  	Let $k = (k_1, \dots, k_s) \in \Z_+^s$ be a partition of $n$ and $\nu,
  	\mu \in \N^n$ a pair of orthogonal multi-indices.  Let $a_1, a_2 \in
  	\mR_{k}$ be non identically zero and $\xi^\nu \overline{\xi}^\mu \in
  	\mH_{k}$. Then, the Toeplitz operators $T_{a_1}$ and
  	$T_{a_2\xi^\nu \overline{\xi}^\mu}$ commute on the Bergman space
    $\mA^2(\Omega_p^n)$
   	if and only if
   	\begin{equation} \label{eq:comm-condition}
	   	\sum_{t=\widehat{k}_{j-1}+1}^{\widehat{k}_{j}}
		\frac{\nu_{t}-\mu_t}{p_{t}} =
		\sum_{t=1}^{k_j} \frac{(\nu_{(j)})_t - (\mu_{(j)})_t}{(p_{(j)})_t}
		=0
	\end{equation}
	for each $j=1,\dots, s$.
\end{theorem}
\begin{proof}
	Let $\alpha\in \Z_+^n$ be given. First note that if $\alpha + \nu - \mu
  	\not\in \Z_+^n$, then the Lemmas~\ref{Lemma-kradial} and
  	\ref{Lemma-khomogeneous} imply that both
  	$T_{a_1}T_{a_2\xi^\nu\overline{\xi}^\mu} z^{\alpha}$ and
  	$T_{a_2\xi^\nu\overline{\xi}^\mu} T_{a_1} z^{\alpha}$ vanish. Hence, we
  	can assume that $\alpha+\nu-\mu \in \Z_+^n$.
  	
  	Applying again Lemmas~\ref{Lemma-kradial} and \ref{Lemma-khomogeneous}
  	we obtain
	\begin{align*}
    	& T_{a_1}T_{a_2\xi^\nu\overline{\xi}^\mu} z^{\alpha} = \\
	    &  \frac{\sty 4^{s}\prod_{t=1}^{n}
	    \Gamma\left(\frac{\alpha_{t}+\nu_{t}+1}{p_{t}}\right)
	    \prod_{j=1}^{s}\left(\sum_{t=\widehat{k}_{j-1}+1}^{\widehat{k}_{j}}
		\frac{1}{p_{t}}\right)
		\Gamma\left(\sum_{j=1}^{n}\frac{\alpha_j+\nu_j-\mu_j+1}{p_j}+1\right)}
		{\sty \prod_{j=1}^{n}
		\Gamma\left(\frac{\alpha_j+\nu_j-\mu_j+1}{p_j}\right)
		\prod_{j=1}^{s} \Gamma\left(\sum_{t=\widehat{k}_{j-1}+1}^{\widehat{k}_{j}}
		\frac{\alpha_{t}+\nu_{t}+1}{p_{t}}\right)}  \\
	 	&\times {\sty\int_{\Delta_{s}^{n}}a_2(r_1,\ldots,r_n)
	 	\prod_{j=1}^{s}r_{j}^{\left(\sum_{t=\widehat{k}_{j-1}+1}^{\widehat{k}_{j}}
		\frac{2\alpha_{t}+\nu_{t}-\mu_t+2}{p_{t}}-1\right)}dr_{j}}  \\
    	&\times \frac{\sty 4^{s}\Gamma\left(\sum_{j=1}^{n}
    	\frac{\alpha_{j}+\nu_j-\mu_j+1}{p_{j}}+1\right)\prod_{j=1}^{s}
    	\left(\sum_{t=\widehat{k}_{j-1}}^{\widehat{k}_{j}}\frac{1}{p_{t}}\right)}
    	{\sty\prod_{j=1}^{s}\Gamma\left(\sum_{t=\widehat{k}_{j-1}+1}^{\widehat{k}_{j}}
    	\frac{\alpha_{t}+\nu_{t}-\mu_{t}+1}{p_{t}}\right)}  \\
    	&\times \int_{\Delta_{s}^{n}} a_1(r_1,\ldots, r_m)\prod_{j=1}^{m}
    	r_{j}^{\left(2\sum_{t=\widehat{k}_{j-1}+1}^{\widehat{k}_{j}}
    	\left(\frac{\alpha_{t}+\nu_t-\mu_t+1}{p_{t}}\right)-1\right)}dr_{j} \\
    	&\times  z^{\alpha+\nu-\mu}.
  	\end{align*}
	And similarly, we have
	\begin{align*}
    	& T_{a_2\xi^\nu\overline{\xi}^\mu}  T_{a_1} z^{\alpha} = \\
    	& \frac{\sty 4^{s}\Gamma\left(\sum_{j=1}^{n}
    	\frac{\alpha_{j}+1}{p_{j}}+1\right)\prod_{j=1}^{s}
    	\left(\sum_{t=\widehat{k}_{j-1}}^{\widehat{k}_{j}}\frac{1}{p_{t}}\right)}
    	{\sty\prod_{j=1}^{s}\Gamma\left(\sum_{t=\widehat{k}_{j-1}+1}^{\widehat{k}_{j}}
    	\frac{\alpha_{t}+1}{p_{t}}\right)}  \\
    	&\times \int_{\Delta_{s}^{n}} a_1(r_1,\ldots, r_m)\prod_{j=1}^{s}
    	r_{j}^{\left(2\sum_{t=\widehat{k}_{j-1}+1}^{\widehat{k}_{j}}
    	\left(\frac{\alpha_{t}+1}{p_{t}}\right)-1\right)}dr_{j}  \\
		&\times \frac{\sty 4^{s}\prod_{t=1}^{n}
		\Gamma\left(\frac{\alpha_{t}+\nu_{t}+1}{p_{t}}\right)
		\prod_{j=1}^{s}\left(\sum_{t=\widehat{k}_{j-1}+1}^{\widehat{k}_{j}}
		\frac{1}{p_{t}}\right)}
		{\sty \prod_{j=1}^{n}
		\Gamma\left(\frac{\alpha_j+\nu_j-\mu_j+1}{p_j}\right)
		}  \\
		&\times \frac{\sty\Gamma\left(\sum_{j=1}^{n}\frac{\alpha_j+\nu_j-\mu_j+1}{p_j}+1\right)}
		{\sty\prod_{j=1}^{s} \Gamma\left(\sum_{t=\widehat{k}_{j-1}+1}^{\widehat{k}_{j}}
				\frac{\alpha_{t}+\nu_{t}+1}{p_{t}}\right)} \\
		&\times
		{\sty\int_{\Delta_{s}^{n}}a_2(r_1,\ldots,r_n)
		\prod_{j=1}^{s}r_{j}^{\left(\sum_{t=\widehat{k}_{j-1}+1}^{\widehat{k}_{j}}
		\frac{2\alpha_{t}+\nu_{t}-\mu_t+2}{p_{t}}-1\right)}dr_{j}}\\
		&\times  z^{\alpha+\nu-\mu}
	\end{align*}
	This implies that $T_{a_1}T_{a_2\xi^\nu\overline{\xi}^\mu} z^\alpha =
	T_{a_2\xi^\nu\overline{\xi}^\mu} T_{a_1} z^\alpha$ for all $\alpha$
	if and only if
	\[
		\sum_{t=\widehat{k}_{j-1}+1}^{\widehat{k}_{j}}
		\frac{\nu_{t}-\mu_t}{p_{t}}=0
	\]
	where $j=1,\dots, s$.
\end{proof}

If we assume that \eqref{eq:comm-condition} holds for all
$j = 1, \dots, s$, then equations \eqref{spectrum-radial}
and \eqref{spectrum-homoge} imply the following identity

\begin{align} \label{spectrum-homoge-2}
	&\tilde{\gamma}_{a,k,\nu,\mu}(\alpha) = \\
	&{\sty \int_{\Delta_{m}^{n}}a(r_1,\ldots,r_n)
	\prod_{j=1}^{s}r_{j}^{\left(\sum_{t=\widehat{k}_{j-1}+1}^{\widehat{k}_{j}}
	\frac{2\alpha_{t}+2}{p_{t}}-1\right)}dr_{j}}  \notag\\
	&\times \frac{\sty 4^{s}\prod_{t=1}^{n}
	\Gamma\left(\frac{\alpha_{t}+\nu_{t}+1}{p_{t}}\right)
	\prod_{j=1}^{s}\left(\sum_{t=\widehat{k}_{j-1}+1}^{\widehat{k}_{j}}
	\frac{1}{p_{t}}\right)
	\Gamma\left(\sum_{j=1}^{n}\frac{\alpha_j+1}{p_j}+1\right)}
	{\sty \prod_{j=1}^{n} \Gamma\left(\frac{\alpha_j+\nu_j-\mu_j+1}{p_j}\right)
	\prod_{j=1}^{s} \Gamma\left(\sum_{t=\widehat{k}_{j-1}+1}^{\widehat{k}_{j}}
	\frac{\alpha_{t}+\nu_{t}+1}{p_{t}}\right)} \notag \\
	&= \frac{\sty\prod_{t=1}^{n}
	\Gamma\left(\frac{\alpha_{t}+\nu_{t}+1}{p_{t}}\right)
	\prod_{j=1}^{s}\Gamma\left(\sum_{t=\widehat{k}_{j-1}+1}^{\widehat{k}_{j}}
	\frac{\alpha_{t}+1}{p_{t}}\right)}
	{\sty \prod_{j=1}^{n} \Gamma\left(\frac{\alpha_j+\nu_j-\mu_j+1}{p_j}\right)
	\prod_{j=1}^{s} \Gamma\left(\sum_{t=\widehat{k}_{j-1}+1}^{\widehat{k}_{j}}
	\frac{\alpha_{t}+\nu_{t}+1}{p_{t}}\right)} \gamma_{a,k}(\alpha) \notag
\end{align}

\begin{theorem}\label{thm:commutative-iff}
	Let $k = (k_1, \dots, k_s) \in \Z_+^s$ be a partition of $n$
  	and let $\nu, \mu, \sigma,\eta \in \N^n$
  	be multi-indices that satisfy the following properties
  	\begin{itemize}
  	\item $\nu \perp \mu$ and $\sigma \perp \eta$,
  	\item ${\sty\sum_{t=\widehat{k}_{j-1}+1}^{\widehat{k}_{j}}
		\frac{\nu_{t}-\mu_t} {p_{t}}=0}$ and
		${\sty\sum_{t=\widehat{k}_{j-1}+1}^{\widehat{k}_{j}}
		\frac{\sigma_{t}-\eta_t}{p_{t}}=0}$ for all $j = 0, \dots, s$.
  	\end{itemize}
	Let $a\xi^\nu\overline{\xi}^\mu, b\xi^\sigma\overline{\xi}^\eta \in
  	\mHR_{k}$ be corresponding $k$-quasi-homogeneous
  	quasi-radial symbols on $\Omega^n_p$, where $a, b \in \mR_{k}$ are
  	measurable and bounded symbols. Then, the Toeplitz operators
  	$T_{a\xi^\nu\overline{\xi}^\mu}$ and
  	$T_{b\xi^\sigma\overline{\xi}^\eta}$ commute on the
  	Bergman space $\mA^2(\Omega_p^n)$ if and only
  	if for each $s=1,\ldots,n$ one of
  	the following conditions holds
  	\begin{enumerate}
  	\item $\nu_s=\mu_s=0$
  	\item $\sigma_s=\eta_s=0$
  	\item $\nu_s=\sigma_s=0$
  	\item $\mu_s=\eta_s=0$
  	\end{enumerate}
\end{theorem}
\begin{proof}
	First, we observe that $T_{b\xi^\sigma\overline{\xi}^\eta}T_{a\xi^\nu\overline{\xi}^\mu} z^{\alpha}$
  	and $T_{a\xi^\nu\overline{\xi}^\mu}T_{b\xi^\sigma\overline{\xi}^\eta} z^{\alpha}$
	are always simultaneously zero or non zero. Hence, we
	compute such expressions for $\alpha \in N^n$ assuming that both
	are non zero.
	By \eqref{spectrum-homoge-2}, we have the following expression
	\begin{align*}
  		& T_{b\xi^\sigma\overline{\xi}^\eta}T_{a\xi^\nu\overline{\xi}^\mu} z^\alpha = \\
  		& {\sty \int_{\Delta_{m}^{n}}a(r_1,\ldots,r_n)
  		\prod_{j=1}^{s}r_{j}^{\left(\sum_{t=\widehat{k}_{j-1}+1}^{\widehat{k}_{j}}
		\frac{2\alpha_{t}+2}{p_{t}}-1\right)}dr_{j}}  \\
		&\times \frac{\sty 4^{s}\prod_{t=1}^{n}\Gamma\left(\frac{\alpha_{t}+\nu_{t}+1}{p_{t}}\right)
		\prod_{j=1}^{s}\left(\sum_{t=\widehat{k}_{j-1}+1}^{\widehat{k}_{j}}
		\frac{1}{p_{t}}\right) \Gamma\left(\sum_{j=1}^{n}\frac{\alpha_j+1}{p_j}+1\right)}
		{\sty \prod_{j=1}^{n}
		\Gamma\left(\frac{\alpha_j+\nu_j-\mu_j+1}{p_j}\right)
		\prod_{j=1}^{s} \Gamma\left(\sum_{t=\widehat{k}_{j-1}+1}^{\widehat{k}_{j}}
		\frac{\alpha_{t}+\nu_{t}+1}{p_{t}}\right)}  \\
		&\times {\sty \int_{\Delta_{m}^{n}}b(r_1,\ldots,r_n)
		\prod_{j=1}^{s}r_{j}^{\left(\sum_{t=\widehat{k}_{j-1}+1}^{\widehat{k}_{j}}
		\frac{2\alpha_{t}+2}{p_{t}}-1\right)}dr_{j}}  \\
		&\times \frac{\sty 4^{s}\prod_{t=1}^{n}
		\Gamma\left(\frac{\alpha_{t}+\nu_t-\mu_t+\sigma_{t}+1}{p_{t}}\right)
		\prod_{j=1}^{s}\left(\sum_{t=\widehat{k}_{j-1}+1}^{\widehat{k}_{j}}
		\frac{1}{p_{t}}\right) }
		{\sty \prod_{j=1}^{n}
		\Gamma\left(\frac{\alpha_j+\nu_j-\mu_j+\sigma_j-\eta_j+1}{p_j}\right)
		}  \\
		&\times \frac{\sty \Gamma\left(\sum_{j=1}^{n}\frac{\alpha_j+1}{p_j}+1\right)}
		{\sty \prod_{j=1}^{s} \Gamma\left(\sum_{t=\widehat{k}_{j-1}+1}^{\widehat{k}_{j}}
				\frac{\alpha_{t}+\sigma_{t}+1}{p_{t}}\right)}
		\times z^{\alpha+\nu-\mu+\sigma-\eta}
  	\end{align*}
	Similarly, we also have
	\begin{align*}
		& T_{a\xi^\nu\overline{\xi}^\mu}T_{b\xi^\sigma\overline{\xi}^\eta} z^{\alpha} = \\
	  	&{\sty \int_{\Delta_{m}^{n}}b(r_1,\ldots,r_n)
	  	\prod_{j=1}^{s}r_{j}^{\left(\sum_{t=\widehat{k}_{j-1}+1}^{\widehat{k}_{j}}
		\frac{2\alpha_{t}+2}{p_{t}}-1\right)}dr_{j}} \\
		&\times \frac{\sty 4^{s}\prod_{t=1}^{n}\Gamma\left(\frac{\alpha_{t}+\sigma_{t}+1}{p_{t}}\right)
		\prod_{j=1}^{s}\left(\sum_{t=\widehat{k}_{j-1}+1}^{\widehat{k}_{j}}
		\frac{1}{p_{t}}\right) \Gamma\left(\sum_{j=1}^{n}\frac{\alpha_j+1}{p_j}
		  +1\right)}
		{\sty \prod_{j=1}^{n}
		\Gamma\left(\frac{\alpha_j+\sigma_j-\eta_j+1}{p_j}\right)
		\prod_{j=1}^{s} \Gamma\left(\sum_{t=\widehat{k}_{j-1}+1}^{\widehat{k}_{j}}
		\frac{\alpha_{t}+\sigma_{t}+1}{p_{t}}\right)} \\
		&\times {\sty \int_{\Delta_{m}^{n}}a(r_1,\ldots,r_n)
		\prod_{j=1}^{s}r_{j}^{\left(\sum_{t=\widehat{k}_{j-1}+1}^{\widehat{k}_{j}}
		\frac{2\alpha_{t}+2}{p_{t}}-1\right)}dr_{j}}  \\
		&\times \frac{\sty 4^{s}\prod_{t=1}^{n}
		\Gamma\left(\frac{\alpha_{t}+\sigma_t-\eta_t+\nu_{t}+1}{p_{t}}\right)
		\prod_{j=1}^{s}\left(\sum_{t=\widehat{k}_{j-1}+1}^{\widehat{k}_{j}}
		\frac{1}{p_{t}}\right) }
		{\sty \prod_{j=1}^{n}
  		\Gamma\left(\frac{\alpha_j+\sigma_j-\eta_j+\nu_j-\mu_j+1}{p_j}\right)
		} \\
		&\times \frac{\sty \Gamma\left(\sum_{j=1}^{n}\frac{\alpha_j+1}{p_j}
					+1\right)}
					{\sty \prod_{j=1}^{s} \Gamma\left(\sum_{t=\widehat{k}_{j-1}+1}^{\widehat{k}_{j}}
							\frac{\alpha_{t}+\nu_{t}+1}{p_{t}}\right)}
		\times z^{\alpha+\nu-\mu+\sigma-\eta}
  	\end{align*}
	Therefore, we conclude that
	  $T_{a\xi^\nu\overline{\xi}^\mu}T_{b\xi^\sigma\overline{\xi}^\eta} z^{\alpha}=
	   T_{b\xi^\sigma\overline{\xi}^\eta}T_{a\xi^\nu\overline{\xi}^\mu}$
	if and only if
	\[
		\prod_{t=1}^{n}
	  	\frac{\Gamma\left(\frac{\alpha_{t}+\nu_{t}+1}{p_{t}}\right)
		\Gamma\left(\frac{\alpha_{t}+\nu_{t}-\mu_t+\sigma_t+1}{p_{t}}\right)}{\Gamma\left(\frac{\alpha_{t}+\nu_{t}-\mu_t+1}{p_{t}}\right)} =
		\prod_{t=1}^{n}
	  	\frac{\Gamma\left(\frac{\alpha_{t}+\sigma_{t}+1}{p_{t}}\right)\Gamma\left(\frac{\alpha_{t}+\sigma_t-\eta_t+\nu_{t}+1}{p_{t}}\right)}{\Gamma\left(\frac{\alpha_{t}+\sigma_{t}-\eta_t+1}{p_{t}}\right)}.
  	\]
	Finally, one can easily check that the latter identity holds for
  	every $\alpha \in  \N^n $ if and only if the conclusion of the
	statement holds. This proves the Theorem.
\end{proof}

The above implies one of our main results: the construction of
a commutative Banach algebra of Toeplitz operators on the
domain $\Omega_p^n$. Note that the rest of this section
generalizes the results found in \cite{NikolaiQuasi}.

\begin{definition} \label{def:commuting-symbols}
	Let $k \in \Z_+^s$ be a partition of $n$ and $h \in \Z_+^s$
	be such that $1 \leq h_j \leq k_j -1$ for all $j = 1, \dots s$.
	Denote by $\mA_{k,h}(\Omega_p^n)$ the set of symbols $\varphi \in \mHR_k(\Omega_p^n)$ that
	satisfy the following properties.
	\begin{enumerate}
		\item The symbol is of the form
			$\varphi = a \xi^\nu \overline{\xi}^\mu$ where
			$a \in \mR_k(\Omega_p^n)$ and $\nu, \mu \in \N^n$ so that
			$\xi^\nu \overline{\xi}^\mu \in \mH_k(\Omega_p^n)$.
		\item The multi-indices $\nu, \mu$ are orthogonal and
			satisfy
			\[
				\sum_{t=\widehat{k}_{j-1}+1}^{\widehat{k}_{j}}
				\frac{\nu_{t}-\mu_t}{p_{t}} =
				\sum_{t=1}^{k_j}
				\frac{(\nu_{(j)})_t - (\mu_{(j)})_t}{(p_{(j)})_t}
				=0
			\]
			for all $j = 1, \dots, s$.
		\item The multi-indices $\nu, \mu$ satisfy
			\[
				\nu_{\widehat{k}_j + t_1} = (\nu_{(j)})_{t_1} = 0, \quad
				\mu_{\widehat{k}_j + t_2} = (\mu_{(j)})_{t_2} = 0
			\]
			for all $1 \leq t_2 \leq h_j < t_1 \leq k_j$ and $j = 1, \dots, s$.
	\end{enumerate}
\end{definition}

\begin{remark}\label{rmk:lcm-Lambda}
	From now on, for $p$ as before we will denote by $\lcm(p)$ the least
	common multiple of $p_1, \dots, p_n$. Also, let us denote
	\[
		\Lambda(p) = \lcm(p) \left(\frac{1}{p_1}, \dots, \frac{1}{p_n}\right),
	\]
	which belongs to $\Z^n_+$. With these conventions, multiplying by $\lcm(p)$
	shows that condition (2) from Definition~\ref{def:commuting-symbols} is
	equivalent to
	\begin{enumerate}
		\item[2'] The multi-indices $\nu, \mu$ are orthogonal and satisfy
			\[
				\Lambda(p)_{(j)} \cdot (\nu_{(j)} - \mu_{(j)})
				= \sum_{t=\widehat{k}_{j-1}+1}^{\widehat{k}_j}
						\Lambda(p)_t (\nu_t - \mu_t) = 0,
			\]
			for every $j = 1, \dots, s$.
	\end{enumerate}
	On the other hand, we observe that the set of symbols $\mA_{k,h}(\Omega_p^n)$ reduces to those
	considered in \cite{NikolaiQuasi} when $p = (1, \dots, 1)$ for which
	we have $\Omega_p^n = \B^n$. 	
\end{remark}

As an immediate consequence of Definition~\ref{def:commuting-symbols}
and Theorem~\ref{thm:commutative-iff} we obtain the following result.

\begin{theorem}[Commutative Banach algebra in $\mA^2(\Omega_p^n)$]
	Let $k \in \Z_+^s$ be a partition of $n$ and $h \in \Z_+^s$
	be such that $1 \leq h_j \leq k_j -1$ for all $j = 1, \dots s$.
	Then, the Banach algebra of Toeplitz operators generated by the
	symbols in $\mA_{k,h}(\Omega_p^n)$ is commutative in the weighted Bergman
	space $\mA^2(\Omega_p^n)$.
\end{theorem}

\section{Bundles of Lagrangian frames and quasi-homogeneous symbols}
We will provide a geometric construction associated to the set
of symbols $\mA_{k,h}(\Omega_p^n)$ introduced in the previous section.
In particular, we will consider fixed $k \in \Z_+^s$ a partition of $n$
and $h \in \Z_+^s$ satisfying the conditions of
Definition~\ref{def:commuting-symbols}. We will also use the notation
introduced in Remark~\ref{rmk:lcm-Lambda} as well as the equivalence
between condition (2') from such remark and condition (2) from
Definition~\ref{def:commuting-symbols}.

We recall that $\Omega_p^n$ is a Reinhardt domain and so the action
\begin{align*}
	\T^n \times \Omega_p^n &\rightarrow \Omega_p^n \\
	(\tau, z) &\mapsto (\tau_1 z_1, \dots, \tau_n z_n)
\end{align*}
is holomorphic. Hence, the action is isometric for the K\"ahler
metric defined by the Bergman kernel of $\Omega_p^n$.

As noted in \cite{QS-Quasi-PC}, the known commutative $C^*$-algebras
of Toeplitz operators have an associated Abelian group of isometries.
We refer to \cite{GQV-disk}, \cite{QS-Proj}, \cite{QV-Ball1} and
\cite{QV-Ball2} for a detailed discussion of such behavior. We have
shown that this property extends to quasi-homogeneous symbols
on the complex projective space $\mathbb{P}^n(\C)$ in a suitable
fashion. To describe the corresponding behavior for $\Omega_p^n$
we consider the following homomorphism built from our
current information.
\begin{align*}
	\pi_p : \T^n &\rightarrow \T^n \\
		t &\mapsto t^{\Lambda(p)} = (t_1^{\Lambda(p)_1},
			\dots, t_n^{\Lambda(p)_n}).
\end{align*}
We observe that $\pi_p$ is a surjective homomorphism of Lie groups.
Moreover, $\pi_p$ is a local diffeomorphism.
This homomorphism allows us to give a simpler description of the
group associated to the symbols that we have considered.

\begin{theorem}[Abelian group associated to $\mA_{k,h}(\Omega_p^n)$]
	\label{thm:assoc-abelian-group}
	Let $k \in \Z_+^s$ be a partition of $n$ and $h \in \Z_+^s$
	be such that the conditions of
	Definition~\ref{def:commuting-symbols} are satisfied. Let $\mT_k(p)$
	be the maximal connected subgroup of $\T^n$ that satisfies
	\begin{itemize}
		\item $\tau \in \mT_k(p)$ if and only if $\varphi(\tau z)
		= \varphi(z)$ for every $\varphi \in \mA_{k,h}(\Omega_p^n)$
		and $z \in \Omega_p^n$.
	\end{itemize}
	Then, $\mT_k(p) = \pi_p(\T^s_k)$, where $\T^s_k$ is the subgroup of
	$\T^n$ given by
	\[
		\T^s_k = \{\tau \in \T^n : \exists \, \omega \in \T^s
		\mbox{ such that } 	\tau_{(j)} = (\omega_j, \dots, \omega_j)
		\;
		\forall\, j = 1, \dots, s	\}.
	\]
	In particular, $\mT_k(p)$ is a closed subgroup of $\T^n$ isomorphic
	to $\T^s$.
\end{theorem}
\begin{proof}
	Since $\pi_p$ is surjective, we can consider the elements of $\T^n$
	to be given in the form $\pi_p(\tau)$ for $\tau \in \T^n$.

	Let $\varphi \in \mA_{k,h}(\Omega_p^n)$ be given as in
	Definition~\ref{def:commuting-symbols}, in the form
	$\varphi = a \xi^\nu \overline{\xi}^\mu$. In coordinates this can
	be written as
	\[
		\varphi(z) = \\ a(|z_{(1)}|, \dots, |z_{(s)}|)
		\prod_{j=1}^{s} \left(\frac{z_{(j)}}{|z_{(j)}|}\right)^{\nu_{(j)}}
				\left(\frac{\overline{z}_{(j)}}{|z_{(j)}|}\right)^{\mu_{(j)}},
	\]
	which yields for every $\pi_p(\tau) \in \T^n$
	\begin{multline*}
		\varphi(\pi_p(\tau)z) = \\
		a(|z_{(1)}|, \dots, |z_{(s)}|)
		\prod_{j=1}^{s} \pi_p(\tau)_{(j)}^{\nu_{(j)}} \overline{\pi_p(\tau)}_{(j)}^{\mu_{(j)}}
		\prod_{j=1}^{s} \left(\frac{z_{(j)}}{|z_{(j)}|}\right)^{\nu_{(j)}}
				\left(\frac{\overline{z}_{(j)}}{|z_{(j)}|}\right)^{\mu_{(j)}}.
	\end{multline*}
	We conclude that $\pi_p(\tau) \in \mT_k(p)$ if and only if
	\begin{equation} \label{eq:mTkp-prod-condition}
		\prod_{j=1}^{s} \left(\tau_{(j)}^{\Lambda(p)_{(j)}}\right)^{{\nu_{(j)}} - \mu_{(j)}}
		=
		\prod_{j=1}^{s} \pi_p(\tau)_{(j)}^{\nu_{(j)}} \overline{\pi_p(\tau)}_{(j)}^{\mu_{(j)}}
		= 1,
	\end{equation}
	for every $\nu, \mu \in \Z_+^n$ that satisfies conditions (2) and (3)
	from Definition~\ref{def:commuting-symbols}. We have used the fact that
	$\pi_p(\tau)_{(j)} = \tau_{(j)}^{\Lambda(p)_{(j)}}$ for every $j = 1, \dots, s$.
	
	For $\tau \in \T^s_k$, let $\omega \in \T^s$ such that $\tau_{(j)} = (\omega_j,
	\dots, \omega_j)$ for all $j$. Then, we have
	\[
		\left(\tau_{(j)}^{\Lambda(p)_{(j)}}\right)^{{\nu_{(j)}} - \mu_{(j)}} =
		\omega^{\Lambda(p)_{(j)}\cdot(\nu_{(j)} - \mu_{(j)})} = 1
	\]
	which vanishes for all $j$ since condition (2') holds. Since $\T^s_k$ is
	connected, this shows that $\pi_p(\T^s_k) \subset \mT_k(p)$.
	By the connectedness of $\mT_k(p)$ and since $\pi_p$ is a local diffeomorphism,
	to conclude that $\pi_p(\T^s_k) = \mT_k(p)$ it is enough to show that
	$\dim(\mT_k(p)) = s$. Notice that the inclusion $\pi_p(\T^s_k) = \mT_k(p)$
	already shows that $s \leq \dim(\mT_k(p))$.

	Let $\pi_p(\tau) \in \mT_k(p)$ be given, and choose $j_0 \in \{1,\dots,s\}$
	and $t_1, t_2 \in \Z^+$ such that $1 \leq t_1 \leq h_{j_0} < t_2 \leq k_{j_0}$.
	For such data let us define
	\begin{align*}
	(\nu_{(j)})_t &=
	\begin{cases}
		(p_{(j_0)})_{t_1} &\mbox{if } j=j_0, t=t_1 \\
		0 &\mbox{otherwise}
	\end{cases}, \\
	(\mu_{(j)})_t &=
	\begin{cases}
		(p_{(j_0)})_{t_2} &\mbox{if } j=j_0, t=t_2 \\
		0 &\mbox{otherwise}
	\end{cases}.
	\end{align*}
	Then, it is easily seen that $\nu, \mu$ satisfy the conditions from
	Definition~\ref{def:commuting-symbols}. This implies that \eqref{eq:mTkp-prod-condition}
	is satisfied for our current $\tau$ and such multi-indices $\nu, \mu$.
	Moreover, \eqref{eq:mTkp-prod-condition} reduces to the identity
	\begin{multline*}
		(\tau_{(j_0)})_{t_1}^{\lcm(p)} (\tau_{(j_0)})_{t_2}^{-\lcm(p)} = \\
		(\tau_{(j_0)})_{t_1}^{(\Lambda(p)_{(j_0)})_{t_1}(p_{(j_0)})_{t_1}}
			(\tau_{(j_0)})_{t_2}^{-(\Lambda(p)_{(j_0)})_{t_2}(p_{(j_0)})_{t_2}} = 1.
	\end{multline*}
	Hence, given the choices of $j_0, t_1, t_2$, we conclude that there exists $\omega \in \T^s$
	such that
	\[
		(\tau_{(j)})_{t}^{\lcm(p)} = \omega_j
	\]
	for every $j = 1, \dots, s$ and $t \in \{1, \dots, k_j \}$. If we consider the
	homomorphism
	\begin{align*}
		\rho : \T^n &\rightarrow \T^n \\
			\tau &\mapsto (\tau_1^{\lcm(p)}, \dots, \tau_n^{\lcm(p)}),
	\end{align*}	
	then we have just proved that $\rho(\mT_k(p)) \subset \T^s_k$. We note that
	$\rho$ is a local diffeomorphism and so we conclude that $\dim(\mT_k(p)) \leq s$.
	This proves that $\dim(\mT_k(p)) = s$ and, as remarked above, it shows that
	$\mT_k(p) = \pi_p(\T^s_k)$.
	
	For the last claim, observe that since $\T^s_k \simeq \T^s$ the group
	$\mT_k(p) = \pi_p(\T^s_k)$ is compact and thus closed in $\T^n$. Furthermore,
	the fact that $\pi_p$ is a local isomorphism implies that $\mT_k(p) \simeq
	\T^s_k \simeq \T^s$ since all the groups involved are compact and Abelian.
\end{proof}

\begin{remark} \label{rmk:unit-ball}
	We note that the group $\mT_k(p)$ does not depend on $h$. For $p_0
	= (1, \dots, 1)$ the domain under consideration is the unit ball
	$\Omega^n_{(1, \dots, 1)} = \B^n$, and in this case $\Lambda(p_0) =
	(1, \dots, 1)$ as well. Hence, $\pi_{p_0} = id_{\T^n}$ is the identity
	homomorphism of $\T^n$ and so $\mT_k(p_0) = \T^s_k$.
\end{remark}

Following the references mentioned above, we
now proceed to construct Lagrangian frames associated to our
symbols in $\mA_{k,h}(\Omega_p^n)$ by using the torus $\mT_k(p)$.

First, we note that the homomorphism $\pi_p$ defined above can be
more generally considered as the homomorphism of $\C^{*n}$ with the
same expression
\begin{align*}
	\pi_p : \C^{*n} &\rightarrow \C^{*n} \\
	\zeta &\mapsto (\zeta_1^{\Lambda(p)_1},
	\dots, \zeta_n^{\Lambda(p)_n}).
\end{align*}
We thus introduce the complexification of $\mT_k(p)$ as given
by
\[
	\mT_k^\C(p) = \pi_p(\C^{*s}_k),
\]
for the following subgroup of $\C^{*n}$
\[
	\C^{*s}_k = \{\zeta \in \C^{*n} : \exists \, \omega \in \C^{*s}
			\mbox{ such that } 	\zeta_{(j)} = (\omega_j, \dots, \omega_j) \;
			\forall \, j = 1, \dots, s	\}.
\]
For the corresponding construction in the case of the projective
space we considered an action of the complexified group. However,
in this case, we have only a local action of $\mT_k^\C(p)$ on
$\Omega_p^n$ given by the expression
\begin{equation} \label{eq:local-action}
	(\pi_p(\zeta),z) \mapsto \pi_p(\zeta) z
	= (\zeta_1^{\Lambda(p)_1} z_1, \dots, \zeta_n^{\Lambda(p)_n} z_n).
\end{equation}
where $(\pi_p(\zeta),z) \in \mT_k^\C(p) \times \Omega_p^n$ is restricted
to those pairs such that $\pi_p(\zeta) z \in \Omega_p^n$. More precisely,
we have the following easy to prove result. It also shows that the
local action of $\mT_k^{\C}(p)$ extends the action of $\mT_k(p)$.

\begin{lemma} \label{lem:local-action-domain}
	With the above notation, the subset of pairs $(\pi_p(\zeta),z) \in
	\mT_k^\C(p) \times \Omega_p^n$ such that $\pi_p(\zeta) z \in \Omega_p^n$
	is an open subset of $\mT_k^\C(p) \times \Omega_p^n$. Such open set contains
	\[
		\pi_p(\B_0 \times \dots \times \B_0) \times \Omega_p^n,
	\]
	where the first factor is the image under
	$\pi_p$ of the product of $n$ copies of $\B_0$, the punctured
	disc $\B \setminus\{0\}$ in $\C$. In particular, the local action
	is defined on $\mT_k(p) \times \Omega_p^n$ where it is given by the
	action of $\mT_k(p)$ on $\Omega_p^n$ considered above.
\end{lemma}

We observe that the local action of $\mT_k^{\C}(p)$ is
locally holomorphic in the sense that for $\zeta \in \mT_k^{\C}(p)$ fixed
the map defined by \eqref{eq:local-action} is locally
holomorphic.

To perform our geometric constructions it is better to work with
a global action. Hence, we consider the corresponding action
of $\mT_k^\C(p)$ on $\C^n$ given by the same expression
\begin{align*}
	\mT_k^\C(p) \times \C^n &\rightarrow \C^n \\
	(\pi_p(\zeta),z) &\mapsto \pi_p(\zeta) z
	= (\zeta_1^{\Lambda(p)_1} z_1, \dots, \zeta_n^{\Lambda(p)_n} z_n).
\end{align*}

It is easy to see that the orbits for the action of $\mT_k^{\C}(p)$ on
$\C^n$ have varying dimensions, so we introduce the following
stratification that collects (some) orbits with the same dimension.
For every $j = 1, \dots, s$ we let
\[
	V_j = \{ z \in \C^n : z_{(1)} = 0, \dots, z_{(j-1)} = 0,
	\text{ and } z_{(j)} \not= 0, \dots, z_{(s)} \not= 0 \}.
\]
It is easy to see that, for every $j$, the set $V_j$ is a complex
submanifold of $\C^n$ that is invariant under the action of $\mT_k^{\C}(p)$.
Recall that an action is free when its stabilizers are trivial.
One can consider free actions by the introduction of the
following subgroups of $\mT_k^{\C}(p)$. We let for every $j = 1, \dots, s$
\[
	G_j = \pi_p(\{ \zeta \in \C^{*s}_k : \zeta_t = 1, \mbox{ for all }
	j =1, \dots, \widehat{k}_{j-1}
	\}),
\]
which defines a descending collection of groups with
$G_1 = \mT_k^\C(p)$. Also note that $G_j \simeq \C^{*(s+1-j)}$ for all $j$.
Consider the subgroups of $\mT_k^{\C}(p)$ given by
\[
	\mT_j = G_j \cap \mT_k(p),
\]
for every $j = 1, \dots, s$. Hence, the subgroups
$\mT_j$ form a descending collection of compact subgroups of $\mT_k(p)$
such that $\mT_1 = \mT_k(p)$ and satisfying $\mT_j \simeq \T^{s+1-j}$.

The following result is an easy consequence of the
definitions and the fact that the equations
$z_{(1)} = 0, \dots, z_{(j-1)} = 0$
define closed complex submanifolds of $\C^n$. Note that the
set $V_1 \cup \dots \cup V_s$ is an open conull dense subset
of $\C^n$.

\begin{lemma} \label{lem:stratification}
	For the partition $V_1 \cup \dots \cup V_s
	\subset \C^n$ and the groups $G_1, \dots, G_s$
	defined above, the following properties hold
	for every $j = 1, \dots, s$.
	\begin{enumerate}
		\item The subset $V_j \cup \dots \cup V_s$ is
			a closed complex submanifold of $\C^n$.
		\item The submanifold $V_j$ is open in
			$V_j \cup \dots \cup V_s$.
		\item The action of the group $G_j$ on
			$\C^n$ leaves invariant $V_j \cup
			\dots \cup V_s$.
		\item The submanifold $V_j$ is the largest
			subset of $V_j \cup \dots \cup V_s$ where
			the local action of $G_j$ is free.
	\end{enumerate}
\end{lemma}

Next we obtain a collection of fiber bundles whose total
spaces are the submanifolds $V_j$ and whose structure groups
are the corresponding $G_j$. We refer to \cite{KNI} for more
details on the notions related to fiber bundles. The following
result thus provides the first geometric structure associated
to the symbols in $\mA_{k,h}(\Omega_p^n)$.

\begin{theorem}[Principal bundles associated to $\mA_{k,h}(\Omega_p^n)$]
	\label{thm:assoc-principal-bundles}
	Let $k \in \Z_+^s$ be a partition of $n$ and $h \in \Z_+^s$
	be such that the conditions of
	Definition~\ref{def:commuting-symbols} are satisfied. Consider the
	submanifolds $V_1, \dots, V_s$ of $\C^n$ and the subgroups
	$G_1, \dots, G_s$ of $\mT_k^{\C}(p)$ defined above. Then, the
	following property is satisfied for every $j = 1, \dots, s$.
	\begin{itemize}
		\item The quotient space $G_j\backslash V_j$ is a complex
			manifold so that the natural quotient map
			$V_j \rightarrow G_j \backslash V_j$ is a complex fiber
			bundle with structure group $G_j \simeq \C^{*(s+1-j)}$.
			In particular, every $G_j$-orbit is a complex submanifold
			of $\C^n$.		
	\end{itemize}
\end{theorem}
\begin{proof}
	By Lemma~\ref{lem:local-action-domain} the action of $G_j$ on $V_j$
	is free and so by the results from \cite{GHL} it is enough to show
	that the action of $G_j$ is proper. Recall that a $G$-action on a
	manifold $V$ is proper if for every compact subset $K \subset V$ the
	set
	\[
		\{g \in G : gK \cap K \not= \emptyset \}
	\]
	has compact closure in $G$.
	
	Let $K \subset V_j$ be a compact subset. From the definition of $V_j$
	it follows we have
	\begin{itemize}
		\item if $z \in K$, then $z_{(t)} \not= 0$ for every $t = j,
			\dots, s$,
		\item there exists $c_1, c_2 > 0$ such that $c_1 \leq
			|z_t| \leq c_2$, for all $t = \widehat{k}_{(j-1)} + 1, \dots, n$,
			and for	every $z \in K$.
	\end{itemize}
	Let $\pi_p(\zeta) \in G_j$ be given, so that in particular we have
	$\zeta_t = 1$ for all $t = 1, \dots, \widehat{k}_{(j-1)}$. If
	$\pi_p(\zeta)K \cap K \not= \emptyset$, then there exists $z, z' \in
	K$ such that
	\[
		\zeta_t^{\Lambda(p)_t} z_t = z'_t
	\]
	for all $t = \widehat{k}_{(j-1)} +1 , \dots, n$. This implies that
	\[
		C^{-1} \leq |\zeta_t|^{\Lambda(p)_t} \leq C
	\]
	for every $t = \widehat{k}_{(j-1)} +1 , \dots, n$, if we take
	$C = c_2/c_1$. Since this defines a condition with compact closure
	in $G_j$, we conclude that the $G_j$-action on $V_j$ is compact.	
\end{proof}

The structure considered in the previous result is easily seen
to induced a corresponding one on the domain $\Omega_p^n$. Its proof
is an easy consequence of Lemma~\ref{lem:stratification} and
Theorem~\ref{thm:assoc-principal-bundles} since $\Omega_p^n$
is open in $\C^n$.

\begin{theorem}
	\label{thm:assoc-fibers}
	Let $k \in \Z_+^s$ be a partition of $n$ and $h \in \Z_+^s$
	be such that the conditions of
	Definition~\ref{def:commuting-symbols} are satisfied. Consider the
	submanifolds $V_1, \dots, V_s$ of $\C^n$ and the subgroups
	$G_1, \dots, G_s$ of $\mT_K^{\C}(p)$ defined above, and denote
	$\widehat{V}_j = V_j \cap \Omega_p^n$ for every $j = 1, \dots, s$.
	Then, the following properties are satisfied for every
	$j = 1, \dots, s$
	\begin{enumerate}
		\item The subset $\widehat{V}_j \cup \widehat{V}_s$
			is a closed complex	submanifold of $\Omega_p^n$.
		\item The submanifold $\widehat{V}_j$ is open in
			$\widehat{V}_j \cup \widehat{V}_s$.
		\item The intersection of the $G_j$-orbits in $V_j$
			with $\Omega_p^n$ defines a foliation $\F_j$
			by complex submanifolds in $\widehat{V}_j$.
	\end{enumerate}
\end{theorem}

We now introduce the notion of a Lagrangian frame that was first
considered in \cite{QV-Reinhardt}, and refer to this work for the
details of the geometric notions involved.

\begin{definition} \label{def:Lag-frame}
	For a K\"ahler manifold $N$, a Lagrangian frame is a pair of
	smooth foliations $(\F,\OO)$ that satisfy the following
	properties
	\begin{enumerate}
		\item Both foliations are Lagrangia. In other words, their
			leaves are Lagrangian submanifolds of $N$.
		\item If $L_1$ and $L_2$ are leaves of $\F$ and $\OO$, respectively,
			then $T_xL_1 \perp T_xL_2$ at every point $x \in L_1 \cap L_2$.
		\item The foliation $\OO$ is Riemannian. In other words, the
			the Riemannian metric of $N$ is invariant under the leaf
			holonomy of $\OO$.
		\item The foliation $\F$ is totally geodesic. In other words,
			the leaves of $\F$ are totally geodesic submanifolds of
			$N$.
	\end{enumerate}
	And in this case, we refer to $\OO$ and $\F$ as the Riemannian
	and totally geodesic foliations, respectively, of the Lagrangian
	frame.
\end{definition}

It is a remarkable fact that the symbols in $\mA_{k,h}(\Omega_p^n)$
yield a Lagrangian frame on $\Omega_p^n$, as the following result
establishes.

\begin{theorem}[Lagrangian frames associated to $\mA_{k,h}(\Omega_p^n)$]
	\label{thm:assoc-Lag-frame}
	Let $k \in \Z_+^s$ be a partition of $n$ and $h \in \Z_+^s$
	be such that the conditions of
	Definition~\ref{def:commuting-symbols} are satisfied. Then,
	for every $j =1, \dots, s$ and for every leaf $L$ in the foliation
	$\F_j$ of $\widehat{V}_j$ the following properties hold.
	\begin{enumerate}
		\item The action of $\mT_j$ restricted to $L$ defines a Riemannian
			foliation $\OO_L$ on whose leaves every symbol in
			$\mA_{k,h}(\Omega_p^n)$ is constant.
		\item The vector bundle $T\OO_L^\perp$ defined as the orthogonal
			complement of $T\OO_L$ in $TL$ is integrable to a totally
			geodesic foliation $\iO_L$.
		\item The pair $(\OO_L, \iO_L)$ is a Lagrangian frame in
			$L$ for the K\"ahler structure inherited from $\Omega_P^n$.
	\end{enumerate}
\end{theorem}
\begin{proof}
	Note that $L$ is an open set of a $G_j$-orbit in $V_j$. Hence, the complex
	dimension of $L$ is $\dim_\C G_j = s+1-j$. Hence, the real dimension of the
	$\mT_j$-orbits in $L$ is $\dim \mT_j = s+1-j$. Also, $\mT_j$ leaves invariant
	both $V_j$ and $\Omega_p^n$, and so it leaves invariant $L$. Since the
	$\mT_j$-action on $V_j$ is free, it is so on $L$ and hence it defines a foliation
	$\OO_L$ in $L$.
	
	Since $\mT_j \subset \T^n$ which acts by isometries on the Reinhardt domain
	$\Omega_p^n$, it follows that the foliation $\OO_L$ is Riemannian in $L$
	for the K\"ahler structure on $L$ inherited from $\Omega_p^n$. We refer
	to \cite{QV-Reinhardt} for the relevant results. We also know from \cite{QV-Reinhardt}
	that the $\T^n$-orbits on $\Omega_p^n$ are Lagrangian. Since $\mT_j$ acts on $L$ by
	a restriction of the $\T^n$-action, it follows that the $\mT_j$-orbits are isotropic
	in $L$; here we have also used that $L$ has the K\"ahler structure inherited from
	that of $\Omega_p^n$. Since the real dimension of the $\mT_j$-orbits in $L$ and
	the complex dimension of $L$ are the same, we conclude that $\OO_L$ is a
	Lagrangian foliation. Also observe that, by Theorem~\ref{thm:assoc-abelian-group},
	the symbols in $\mA_{k,h}(\Omega_p^n)$ are constant on the leaves of the
	foliation $\OO_L$.
	
	Since the orthogonal complement to the tangent bundle of a Riemannian
	foliation is totally geodesic (see \cite{QV-Reinhardt}), to complete the
	proof it is enough to show that $T\OO^\perp_L = i T\OO_L$ is an integrable
	vector subbundle of $TL$.
	
	To prove the integrability of $iT\OO_F$ let us consider $\{ X_t : t
	= 1, \dots, s+1-j \}$ a base for the Lie algebra of $\mT_j$. The
	$\mT_j$-action on $L$ induces a family of vector fields $\{ X_t^* :
	t = 1, \dots, s+1-j \}$ on $L$ characterized as those having flows
	given by $\{ \exp(X_t) : t = 1, \dots, s+1-j \}$, respectively. We
	refer to \cite{Helgason} for further details on this
	construction. Since the $\mT_j$-action on $L$ is free, it follows
	that the vector fields $\{ X_t^* : t = 1, \dots, s+1-j \}$ are
	linearly independent at every point of $L$, thus defining a
	generating set for $T\OO_L$ at every point of $L$. Furthermore,
	since $\mT_j$ is Abelian $[X^*_{t_1}, X^*_{t_2}] = 0$ for every
	$t_1, t_2$.
	
	On the other hand, for $J$ the complex structure of $\Omega_p^n$,
	the set of vector fields $\{ JX_t^* : t = 1, \dots, s+1-j \}$
	yields a generating set	for $iT\OO_F$ at every point of $F$.
	
	We claim that for every $t$, the vector fields $X^*_t$ and $JX^*_t$
    are holomorphic, i.e.~they integrate to holomorphic local flows.
	First note that, dy definition, the vector fields $X^*_t$
	to flows which are $1$-parameter subgroups of the $\mT_j$-action.
	Since the latter is holomorphic, we	conclude that the vector fields $X^*_t$
	are holomorphic.
	
    We now consider the vector fields $JX^*_t$. First we recall that
	$\mT_j \subset \mT_k(p) \subset \mT_k^\C(p)$, which implies that
	\[
		X_t \in \Lie(\mT_j) \subset \Lie(\mT_k^\C(p))
	\]
	Since the local action of $\mT_k^\C(p)$ on $\Omega_p^n$ is holomorphic
	we conclude that
	\[
		JX_t^* = (iX_t)^*
	\]
	In particular, $JX_t^*$ is a holomorphic vector field on $\Omega_p^n$.
	Once this is known, we have for every $t_1, t_2$
	\[
		[JX^*_{t_1}, JX^*_{t_2}] = J[X^*_{t_1}, JX^*_{t_2}]
		= J^2[X^*_{t_1}, X^*_{t_2}] = 0.
	\]
	Here we have used in the first and second identities the fact that
	$JX^*_{t_2}$ and $X^*_{t_1}$, respectively, define Lie derivatives that
	commute with $J$; the latter is a consequence of the fact that both
	vector fields are holomorphic (see \cite{KNII}). Thus, we have
	proved that the bundle $iT\OO_p$ has a set of sections that generate
	the fibers and commute pairwise. Hence, the integrability of
	$iT\OO_p$ follows from Frobenius Theorem.
\end{proof}

Finally, we mention that a result similar to Theorem~5.9 from \cite{QS-Quasi-PC}
can be obtained by considering a suitable complement of $\mT_k(p)$ in$\mT_k^\C(p)$.
However, this provides only a local action on $\Omega_p^n$ and so we omit the
details of such construction.

\end{document}